\documentclass[a4paper]{article}
\usepackage{amssymb,amsmath}
\usepackage{amsthm}
\usepackage{graphicx}
\usepackage{url}
\usepackage{algorithm,algorithmic}
\begin{document}

%================================================
\newtheorem{Definition}{Definition}
\newtheorem{Theorem}[Definition]{Theorem}
\newtheorem{Lemma}[Definition]{Lemma}
\newtheorem{Corollary}[Definition]{Corollary}
\newtheorem{Example}[Definition]{Example}
\newtheorem{Remark}[Definition]{Remark}
\newtheorem{Notation}[Definition]{Notation}
\newtheorem{Proposition}[Definition]{Proposition}
\renewcommand{\algorithmicrequire}{\textbf{Input:}}
\renewcommand{\algorithmicensure}{\textbf{Output:}}
\newcommand{\Ext}{{\mathop{\mathrm{Ext}}}}
\newcommand{\Ann}{{\mathop{\mathrm{ann}}}}
\newcommand{\Ass}{{\mathop{\mathrm{Ass}}}}
\newcommand{\hull}{{\mathop{\mathrm{hull}}}}
\newcommand{\codim}{{\mathop{\mathrm{codim}}}}
\newcommand{\aminor}{{\mathop{\mathrm{aminor}}}}
%%%%%%%%%%%%%%%%%%%%%%%%%%%%%

%========================================

\title{Effective Localization Using Double Ideal Quotient and Its Implementation} 

\author{Yuki Ishihara\thanks{Graduate School of Science, Rikkyo University 3-34-1 Nishi-Ikebukuro, Toshima-ku, Tokyo, Japan 171-8501, yishihara@rikkyo.ac.jp} \and Kazuhiro Yokoyama\thanks{Department of Mathematics, Rikkyo University 3-34-1 Nishi-Ikebukuro, Toshima-ku, Tokyo, Japan 171-8501, kazuhiro@rikkyo.ac.jp}}
\date{}
\maketitle

%==== Abstract ====================================

\begin{abstract}
In this paper, we propose a new method for localization of polynomial ideal, which we call "Local Primary Algorithm". For an ideal $I$ and a prime ideal $P$, our method computes a $P$-primary component of $I$ after checking if $P$ is associated with $I$ by using {\it double ideal quotient} $(I:(I:P))$ and its variants which give us a lot of information about localization of $I$.  
\end{abstract}

%==== Section 1 ====================================

\section{Introduction}\label{sec-1}
In commutative algebra, the operation of “localization by a prime ideal” is well-known as a basic tool. To realize it on computer algebra systems, we propose new effective localization using {\it double ideal quotient} (DIQ) and its variants for ideals, in a polynomial ring over a field. Here, by the words {\em localization}, we mean the saturation or the contraction of localized ideals. 

It is well-known that localization of ideal can be computed through its primary decomposition. In more detail, for an ideal $I$ of a polynomial ring $K[X]=K[x_1,\ldots,x_n]$ over a field $K$ and a multiplicatively closed set $S$ in $K[X]$, once a primary decomposition $\mathcal{Q}$ of $I$ is known, the localization (i.e. the contraction of localized ideal) of $I$ by $S$ can be computed by $IK[X]_S\cap K[X]=\bigcap_{Q\in \mathcal{Q},Q\cap S=\emptyset } Q$ (see Remark \ref{rm:locpri}). Algorithms of primary decomposition have been much studied, for example, by  \cite{GIANNI1988149},  \cite{Eisenbud1992}, \cite{SHIMOYAMA1996247} and \cite{KAWAZOE20111158}. However, in practice, as such primary decomposition tends to be very time-consuming, use of primary decomposition is not an efficient way and we need an efficient {\em direct} method without primary decomposition. Toward a direct method of localization, for a given ideal $I$ and a prime ideal $P$, first we provide several criteria for checking if a primary ideal $Q$ can be a $P$-primary component of $I$, and then present a direct method named {\em Local Primary Algorithm} (LPA) which computes a $P$-primary component of $I$. Our method applies different procedures for two cases; isolated and embedded. Both cases use {\it double ideal quotient and its variants} as a tool for generating and checking primary components. Of course, if we know all associated primes disjoint from a multiplicatively closed set, we get its localization without computing other primary components. 

For ideals $I$ and $J$, we call an ideal $(I:(I:J))$ {\it double ideal quotient} in the paper. Double ideal quotient appears in \cite{vasconcelos2004computational} to check associated primes or compute equidimensional hull, and in \cite{Eisenbud1992}, to compute equidimensional radical. We survey other properties of double ideal quotient and find that it and its variants have useful information about localization. For instance, for ideals $I$, $J$ and a primary decomposition $\mathcal{Q}$ of $I$, a variant of DIQ $(I:(I:J)^{\infty})$ coincides with $\bigcap_{Q\in \mathcal{Q},J\subset IK[X]_{\sqrt{Q}}\cap K[X]} Q$. 

To check the practicality of criteria on LPA , we made an implementation on the computer algebra system Risa/Asir \cite{risaasir} and demonstrate the performance in several examples. To evaluate effectiveness coming from its speciality, we compare timings of it to ones of a general algorithm of primary decomposition in Risa/Asir. 

For practical implements we devise several efficient techniques for improving our LPA. (For efficient computation of ideal quotient and saturation, see \cite{vasconcelos2004computational} and \cite{greuel2002singular}). First, instead of computing the {\em equidimensional hull} $\hull (I+P^m)$, we use $\hull (I+P_G^{[m]})$ where $P_G^{[m]}=(f_1^m,\ldots,f_r^m)$ for some generator $G=\{f_1,\ldots,f_r\}$ of $P$. Second, we use {\em a maximal independent set} of $P$ for computing  $\hull (\overline{Q})$ where $\overline{Q}$ is a $P$-hull-primary ideal. Since a maximal independent set $U$ of $P$ is one of $I+P^m$, we obtain $\hull (I+P^m)=(I+P^m)K[X]_{K[U]^{\times}}\cap K[X]$. Moreover, we also use $U$ at the first step of LPA; use $IK[X]_{K[U]^{\times}}\cap K[X]$ instead of $I$.  By these efficient techniques, our experiment shows certain practicality of our direct localization method.

%==== Section 2 ====================================

\section{Mathematical Basis}\label{sec-2}
Throughout this paper, we denote a polynomial ring $K[x_1,\ldots,x_n]$ by $K[X]$, where $K$ is a computable field (e.g. the rational field $\mathbb{Q}$ or a finite field $\mathbb{F}_p$) and we denote the set of variables $\{x_1,\ldots,x_n\}$ by $X$. We write $(f_1,\ldots,f_t)_{K[X]}$ for the ideal generated by elements $f_1,\ldots,f_t$ in $K[X]$. If the ring is obvious, we simply use $(f_1,\ldots,f_t)$. When we simply say $I$ is an ideal, it means the $I$ is an ideal of $K[X]$. Moreover, we denote the radical of $I$ by $\sqrt{I}$. 

\subsection{Definition of Primary Decomposition and Localization}

Here we give the definition of primary decomposition and that of localization which seem slightly different from {\em standard} ones. We also give fundamental notions and properties related to localization.

\begin{Definition}
	Let $I$ be an ideal of $K[X]$. A set $\mathcal{Q}$ of primary ideals is called a general primary decomposition of $I$ if $I=\bigcap_{Q\in \mathcal{Q}} Q$. A general primary decomposition $\mathcal{Q}$ is called a primary decomposition of $I$ if the decomposition $I=\bigcap_{Q\in \mathcal{Q}} Q$ is an irredundant decomposition. For a primary decomposition of I, each primary ideal is called a primary component of $I$. The prime ideal associated with a primary component of $I$ is called a prime divisor of $I$ and among all prime divisors, minimal prime ideals are called isolated prime divisors of $I$ and others are called embedded prime divisors of $I$. A primary component of $I$ is called isolated if its prime divisor is isolated and embedded if its prime divisor is embedded. We denote by $\Ass(I)$ and $\Ass_{iso} (I)$ the set of all prime divisors of $I$ and the set of all isolated prime divisors respectively. 
\end{Definition}

\begin{Definition}
	Let $I$ be an ideal of $K[X]$ and $S$ a multiplicatively closed set in $K[X]$. We denote the set $\{f\in K[X] \mid fs \in I \text{ for some } s\in  S\}$  by $IK[X]_S\cap K[X]$, and call it the localization of $I$ with respect to $S$. For a multiplicatively closed set $K[X]\setminus P$, where $P$ is a prime ideal, we denote simply by $IK[X]_P\cap K[X]$. We assume a multiplicatively closed set $S$ always does not contain $0$. 
\end{Definition}
	
	\begin{Remark} \label{rm:locpri}
		Given a primary decomposition $\mathcal{Q}$ of an ideal $I$, the localization of $I$ by $S$ is expressed as $\bigcap_{Q\in \mathcal{Q},Q\cap S=\emptyset } Q$. Moreover, it is also equal to $(I:(\bigcap_{P\in \Ass (I),P\cap S\neq \emptyset } P)^{\infty}) $. Thus if we know all primary components or all associated primes, then we can compute localizations of $I$ for any {\em computable} multiplicatively closed sets $S$. (We are thinking mainly about  cases where $S$ is finitely generated or the complement of a prime ideal. In these cases, we can decide efficiently whether $Q$ and $S$ intersect or not). However, this method is not a direct method since it computes unnecessary primary components or associated primes. 
	\end{Remark}
	
	\begin{Lemma}\label{prilocal}
		Let $I$ be an ideal and P a prime divisor of $I$. If $S$ is a multiplicatively closed set with $P\cap S=\emptyset $ and $Q$ is a $P$-primary ideal, then the following conditions are equivalent. 
		
			$(A)$ $Q$ is a primary component of $I$

			$(B)$ $Q$ is a primary component of $IK[X]_S\cap K[X]$
	\end{Lemma}
	
	\begin{proof}
		First, $(A)$ implies $(B)$ from Proposition 4.9 in \cite{atiyah1994introduction} . For primary decompositions $\mathcal{Q}$ of $I$ and $\mathcal{Q}^{\prime}$ of $IK[X]_S\cap K[X]$ with $Q\in \mathcal{Q}^{\prime}$, we obtain $\{Q^{\prime} \in \mathcal{Q}\mid Q^{\prime}\cap S\neq \emptyset \}\cup \mathcal{Q}^{\prime}$ is also a primary decomposition of $I$. Hence, $(B)$ implies $(A)$. 
	\end{proof}
	
		\begin{Definition}[\cite{atiyah1994introduction}, Chapter 4]
		Let $I$ be an ideal. A subset $\mathcal{P}$ of $\Ass (I)$ is said to be {\rm isolated} if it satisfies the following condition: for a prime divisor $P^{\prime}\in \Ass (I)$, if $P^{\prime}\subset P$ for some $P\in \mathcal{P}$, then $P^{\prime}\in \mathcal{P}$. 
	\end{Definition}
	
	\begin{Lemma}[\cite{atiyah1994introduction}, Theorem 4.10] \label{isomul}
		Let $I$ be an ideal and $\mathcal{P}$ an isolated set contained in $\Ass (I)$. For a multiplicatively closed set $S=K[X]\setminus \bigcup_{P\in \mathcal{P}} P$ and a primary decomposition $\mathcal{Q}$ of $I$, $IK[X]_S\cap K[X]=\bigcap_{Q\in \mathcal{Q},\sqrt{Q}\in \mathcal{P}} Q$. 
	
	\end{Lemma}
	
	\begin{Lemma} \label{mulp}
		Let $\mathcal{Q}$ be a primary decomposition of $I$ and $Q\in \mathcal{Q}$. For a multiplicatively closed set $S$, the following conditions are equivalent. 
		
			$(A)$  $IK[X]_S\cap K[X] \subset IK[X]_{\sqrt{Q}}\cap K[X]$.
			
			$(B)$ $Q\cap S=\emptyset$.
	\end{Lemma}
	\begin{proof}
			Show $(A)$ implies $(B)$. As $IK[X]_{\sqrt{Q}}\cap K[X]\subset Q$, $IK[X]_S\cap K[X]=\bigcap_{Q^{\prime}\in \mathcal{Q}, Q^{\prime}\cap S=\emptyset } Q^{\prime}\subset Q$. Since $\mathcal{Q}$ is irredundant,  $IK[X]_S\cap K[X]$ has $\sqrt{Q}$-primary component. Thus, $Q\cap S=\emptyset $. Now, we show $(B)$ implies $(A)$. Then, $\sqrt{Q}\cap S=\emptyset $ and $Q^{\prime}\cap S=\emptyset $ for any $Q^{\prime}\in \mathcal{Q}$ s.t. $Q^{\prime}\subset \sqrt{Q}$. Thus, $IK[X]_{\sqrt{Q}}\cap K[X]=\bigcap_{Q^{\prime}\subset \sqrt{Q}} Q^{\prime}$ implies $IK[X]_S\cap K[X] \subset IK[X]_{\sqrt{Q}}\cap K[X]$. 
	\end{proof}
	
	Next we introduce the notion of pseudo-primary ideal.
	
	\begin{Definition} \label{pseudoprimary}
		Let $Q$ be an ideal. We say $Q$ is {\em pseudo-primary} if $\sqrt{Q}$ is a prime ideal. In this case, we also say $\sqrt{Q}$-pseudo-primary. 
		
	\end{Definition}
	\begin{Definition}	
		Let $I$ be an ideal and $P$ an isolated prime divisor of $I$. For $\mathcal{P}=\{P^{\prime}\in \Ass (I)\mid P$ is the unique isolated prime divisor contained in $P^{\prime}\}$ and $S=K[X]\setminus \bigcup_{P^{\prime}\in \mathcal{P}} P^{\prime}$, we call $\overline{Q}=IK[X]_S\cap K[X]$ the $P$-pseudo-primary component of $I$. This definition is consistent with one in \cite{SHIMOYAMA1996247}. We note that the $P$-pseudo-primary component is determined uniquely and has the $P$-isolated primary component of $I$ as component. 
	\end{Definition}

	\begin{Remark}
		Every $P$-pseudo-primary component of $I$ is a $P$-pseudo-primary ideal. Let $\overline{Q}_P$ is the $P$-pseudo-primary component of $I$. Then $I=\bigcap_{P\in \Ass_{iso} (I)} \overline{Q}_P$ $\cap I^{\prime}$ for some $I^{\prime}$ s.t. $\Ass_{iso} (I^{\prime})\cap \Ass_{iso} (I)=\emptyset $. This decomposition is called a pseudo-primary decomposition in \cite{SHIMOYAMA1996247}, where it is computed by separators from given $\Ass_{iso} (I)$. Meanwhile, we introduce another method to compute it by using double ideal quotient in Lemma \ref{satpseudo}. 
	\end{Remark}
	
	\begin{Definition}
		Let $I$ be an ideal and $\mathcal{Q}$ a primary decomposition of $I$. We call $\hull (I)=\bigcap_{Q\in \mathcal{Q},\dim (Q)=\dim (I)} Q$ the equidimensional hull of $I$. Since every primary component $Q$ satisfying $\dim (Q)=\dim (I)$ is isolated, $\hull (I)$ is determined independently from choice of primary decompositions. 
	\end{Definition}
	
	For a given $I$, $\hull (I)$ can be computed in several manners. For instance, it can be computed by $\Ext$ functors \cite{Eisenbud1992} or a regular sequence contained in $I$ \cite{vasconcelos2004computational}.  
	
	\begin{Proposition}[\cite{Eisenbud1992}, Theorem 1.1. \cite{vasconcelos2004computational}, Proposition 3.41] \label{comhull}
		Let $I$ be an ideal and $u\subset I$ be a $c$-length regular sequence, where $c$ is the codimension of $I$. Then $\hull (I)=((u):((u):I))=\Ann_{K[X]} (\Ext_{K[X]} ^{c} (K[X]/I,K[X]))$. 
	\end{Proposition}
	
	\begin{Definition}
		Let $I$ be an ideal. We say that $I$ is {\em hull-primary} if $\hull (I)$ is a primary ideal. For a prime ideal $P$, we say a hull-primary ideal $I$ is $P$-hull-primary if $P=\hull (\sqrt{I})$. 
	\end{Definition}
	
	Since a pseudo-primary ideal has the unique isolated component, we obtain the following remark. 
	
	\begin{Remark} \label{hull-prima-rmk}
		A pseudo-primary ideal is hull-primary.  
	\end{Remark}
	
	By the definition of the $P$-pseudo-primary component of $I$, it is easy to prove the following lemma. 
		
	\begin{Lemma} \label{psuedo-primary-component}
		Let $P$ be an isolated prime divisor of $I$ and $\overline{Q}$ a $P$-pseudo-primary component of $I$. Then, $\overline{Q}$ is a $P$-hull-primary and $\hull (\overline{Q})$ is the isolated $P$-primary component of $I$. 
	\end{Lemma}
	
	Using Lemma \ref{psuedo-primary-component} and a variant of {\em double ideal quotient}, we generate the isolated $P$-primary component of $I$ in Section 5. 
	
		\begin{Lemma} \label{primaryavoid}
		Let $Q$ be a primary ideal. Let $I$ and $J$ be ideals. If $IJ\subset Q$ and $J\not \subset \sqrt{Q}$, then $I\subset Q$. In particular, if $I\cap J\subset Q$ and $J\not \subset \sqrt{Q}$, then $I\subset Q$. 
		
	\end{Lemma}
	\begin{proof}
		Let $f\in I$ and $g\in J\setminus \sqrt{Q}$. Since $Q$ is $\sqrt{Q}$-primary, $fg\in IJ\subset Q$ and thus $f\in Q$.  
	\end{proof}

	\begin{Lemma} \label{hull-primary subset}
		Let $I$ be a $P$-hull-primary and $Q$ a $P$-primary ideal. If $I\subset Q$, then $\hull (I)\subset Q$. 
	\end{Lemma}
	\begin{proof}
		Let $\mathcal{Q}$ be a primary decomposition of $I$ and $J=\bigcap_{Q^{\prime}\in \mathcal{Q},Q^{\prime}\neq \hull (I)} Q^{\prime}$. Then $I=\hull (I)\cap J\subset Q$ and $J\not \subset P$. Since $Q$ is $P$-primary,  we obtain $\hull (I)\subset Q$ by Lemma \ref{primaryavoid}. 
	\end{proof}
	
	Finally, we recall the famous Prime Avoidance Lemma. 

	\begin{Lemma}[\cite{atiyah1994introduction}, Proposition 1.11] \label{primeavoid}
			{\rm (i)} Let $P_1,\ldots,P_n$ be prime ideals and let $I$ be an ideal contained in $\bigcup_{i=1}^n P_i$. Then, $I\subset P_i$ for some $i$. \\
			{\rm (ii)} Let $I_1,\ldots,I_n$ be ideals and let $P$ be a prime ideal containing $\bigcap_{i=1}^n I_i$. Then $P\supset I_i$ for some $i$. If $P=\bigcap_{i=1}^n I_i$, then $P=I_i$ for some $i$.  
	\end{Lemma}

\subsection{Fundamental Properties of Ideal Quotient}

	We introduce fundamental properties of ideal quotient. The first two  can be seen in several papers and books (\cite{atiyah1994introduction}, Lemma 4.4. \cite{greuel2002singular}, Lemma 4.1.3. \cite{vasconcelos2004computational}, a remark before Proposition 3.56). The last two are direct consequences of the first two. 

	\begin{Lemma} \label{cri2}
		Let $I$ and $J$ be ideals, $Q$ a primary ideal and $\mathcal{Q}$ a primary decomposition of $I$. Then, 
		\begin{align*}
		(Q:J)&=
		\begin{cases}
			Q\text{, if }J\not \subset \sqrt{Q},\\
			K[X]\text{, if }J\subset Q, \\
			\sqrt{Q}\text{-primary ideal properly containing $Q$} \text{, if } J\not \subset Q, J \subset \sqrt{Q},
		\end{cases} \\
		(Q:J^{\infty})&=(Q:\sqrt{J}^{\infty})=
		\begin{cases}
			Q\text{, if }J\not \subset \sqrt{Q},\\
			K[X]\text{, if }J\subset \sqrt{Q},
		\end{cases}\\
		(I:J)&=\bigcap_{Q\in \mathcal{Q},J\not \subset \sqrt{Q}} Q\cap \bigcap_{Q\in \mathcal{Q},J\not \subset Q,J\subset \sqrt{Q}} (Q:J),\\
		(I:J^{\infty})&=(I:\sqrt{J}^{\infty})=\bigcap_{Q\in \mathcal{Q},J\not \subset \sqrt{Q}} Q.
		\end{align*}
		
	\end{Lemma}

%==== Section 3 ====================================

\section{Double Ideal Quotient}
	Double Ideal Quotient (DIQ) is an ideal of shape $(I:(I:J))$ where $I$ and $J$ are ideals. For an ideal $I$ and its primary decomposition  $\mathcal{Q}$, we divide $\mathcal{Q}$ into three parts: 
	\begin{center}	
	$\mathcal{Q}_1(J)=\{Q\in \mathcal{Q}\mid J\not \subset \sqrt{Q}\}, \qquad\mathcal{Q}_2(J)=\{Q\in \mathcal{Q}\mid J \subset Q\},$ \\
	$\mathcal{Q}_3(J)=\{Q\in \mathcal{Q}\mid J\not \subset Q, J\subset \sqrt{Q}\}.$
	\end{center}
	Then, our DIQ is expressed precisely by components of them. The following proposition can be proved directly from Lemma \ref{cri2}. We omit an easy but tedious proof. 

	\begin{Proposition} Let $I$ and $J$ be ideals. Then, \label{dbq}
			\begin{align*}
			(I:(I:J))&=\bigcap_{Q\in \mathcal{Q}_2(J)} \left(Q :\bigcap_{Q^{\prime}\in \mathcal{Q}_1(J)} Q^{\prime}  \cap \bigcap_{Q^{\prime}\in \mathcal{Q}_3(J)} (Q^{\prime}:J)\right)\\ 
				&\cap \bigcap_{Q\in \mathcal{Q}_3(J)}\left(Q :\bigcap_{Q^{\prime}\in \mathcal{Q}_1(J)} Q^{\prime}  \cap \bigcap_{Q^{\prime}\in \mathcal{Q}_3(J)} (Q^{\prime}:J)\right), \\
		\sqrt{(I:(I:J))}&=\bigcap_{P\in \Ass (I),J\subset P} P. 
	\end{align*}
	\end{Proposition}
	
	This proposition can be used to prove the following for prime divisors. 
	
	\begin{Corollary}[\cite{vasconcelos2004computational}, Corollary 3.4] \label{DIPC}
		Let $I$ be an ideal and $P$ a prime ideal. Then, $P$ belongs to $\Ass(I)$ if and only if $P \supset  (I : (I : P))$.
			
	\end{Corollary}
	\begin{proof}
		We note $P\supset (I:(I:P))$ if and only if $P\supset \sqrt{(I:(I:P))}$.  By Proposition \ref{dbq}, $\sqrt{(I:(I:P))}=\bigcap_{P^{\prime}\in \Ass (I), P\subset P^{\prime}} P^{\prime}$. If $P\in \Ass (I)$, then $\sqrt{(I:(I:P))}=\bigcap_{P^{\prime}\in \Ass (I), P\subset P^{\prime}} P^{\prime}\subset P$. On the other hand, if $P\supset \sqrt{(I:(I:P))}$, then there is $P^{\prime}\in \Ass (I)$ s.t. $P^{\prime}\subset P$ and $P^{\prime}\supset P$. Thus $P=P^{\prime}\in \Ass (I)$. 
	\end{proof}
	
	Replacing ideal quotient with saturation in DIQ, we have the following. 
	
	\begin{Proposition}\label{satcolon}
		Let $\mathcal{Q}$ be a primary decomposition of $I$. Then, 
		\begin{align*}
		(I:(I:J)^{\infty})&=\bigcap_{Q\in \mathcal{Q}, J\subset IK[X]_{\sqrt{Q}}\cap K[X]} Q \tag{1},\\
		(I:(I:J^{\infty})^{\infty})&=\bigcap_{Q\in \mathcal{Q}, J\subset \sqrt{IK[X]_{\sqrt{Q}}\cap K[X]}} Q, \tag{2}\\
		(I:(I:J^{\infty}))&=\bigcap_{Q\in \mathcal{Q}_2(J)} (Q:\bigcap_{Q^{\prime}\in \mathcal{Q}_1(J)} Q^{\prime} )\cap  \bigcap_{Q\in \mathcal{Q}_3(J)}(Q :\bigcap_{Q^{\prime}\in \mathcal{Q}_1(J)} Q^{\prime} ). \tag{3}
		\end{align*}
		 We call them {\rm the first saturated quotient}, {\rm the second saturated quotient}, and {\rm the third saturated quotient} respectively. 
	\end{Proposition}
	
	\begin{proof}
		Here, we give an outline of the proof. The formula (1) can be proved by combining the equation 
		\begin{center}
		$(I:(I:J)^{\infty})=(I:\sqrt{(I:J)}^{\infty})=\bigcap_{Q\in \mathcal{Q},\bigcap_{Q^{\prime}\in \mathcal{Q}_1(J)}\sqrt{Q^{\prime}} \cap \bigcap_{Q^{\prime}\in \mathcal{Q}_3(J)}\sqrt{Q^{\prime}}\not \subset \sqrt{Q}} Q$ 
		\end{center}
		by Lemma \ref{cri2} and the following equivalence
		\begin{enumerate}  \setlength{\leftskip}{.5cm} 
			\item[(1-a)] $J\subset IK[X]_{\sqrt{Q}}\cap K[X]$.
			\item[(1-b)] $\bigcap_{Q^{\prime}\in \mathcal{Q}_1(J)}\sqrt{Q^{\prime}} \cap \bigcap_{Q^{\prime}\in \mathcal{Q}_3(J)}\sqrt{Q^{\prime}}\not \subset \sqrt{Q}$. 
		\end{enumerate}
		 for each $Q\in \mathcal{Q}$. The second formula (2) can be proved by combining the equation $(I:(I:J^{\infty})^{\infty})=(I:(I:J^m)^{\infty})=\bigcap_{Q\in \mathcal{Q}, J^m\subset IK[X]_{\sqrt{Q}}\cap K[X]} Q$ for a sufficiently large $m$ from the first formula (1), and the following equivalence 
		\begin{enumerate} \setlength{\leftskip}{.5cm}
			\item[(2-a)] $J^m\subset IK[X]_{\sqrt{Q}}\cap K[X]$ for a sufficiently large $m$.
			\item[(2-b)] $J\subset \sqrt{IK[X]_{\sqrt{Q}}\cap K[X]}$.
		\end{enumerate}
		for each $Q\in \mathcal{Q}$.  The third formula (3) can be proved directly from Lemma \ref{cri2}. 
		
		Now, we explain some details. We show (1-a) implies (1-b). If 
		\begin{center}
		$\bigcap_{Q^{\prime}\in \mathcal{Q}_1(J)}\sqrt{Q^{\prime}} \cap \bigcap_{Q^{\prime}\in \mathcal{Q}_3(J)}\sqrt{Q^{\prime}} \subset \sqrt{Q}$, 
		\end{center}
		then by Lemma \ref{primeavoid}, $\sqrt{Q^{\prime}}\subset \sqrt{Q}$ for some $Q^{\prime}\in \mathcal{Q}_1(J)\cup \mathcal{Q}_3(J)$. Since $Q^{\prime}\subset \sqrt{Q^{\prime}}\subset \sqrt{Q}$, we obtain $IK[X]_{\sqrt{Q}}\cap K[X]=\bigcap_{Q^{\prime \prime}\in \mathcal{Q},Q^{\prime \prime}\subset \sqrt{Q}} Q^{\prime \prime}\subset Q^{\prime}.$
		However, since $Q^{\prime}\in \mathcal{Q}_1(J)\cup \mathcal{Q}_3(J)$, we obtain $J\not \subset Q^{\prime}$ and this contradicts $J\subset IK[X]_{\sqrt{Q}}\cap K[X]\subset Q^{\prime}$. 
		
		Show (1-b) implies (1-a). Let $Q^{\prime}\in \mathcal{Q}$ contained $\sqrt{Q}$. Since $\bigcap_{Q^{\prime \prime}\in \mathcal{Q}_1(J)}\sqrt{Q^{\prime \prime}} \cap \bigcap_{Q^{\prime \prime}\in \mathcal{Q}_3(J)}\sqrt{Q^{\prime \prime}}\not \subset \sqrt{Q}$, we obtain $Q^{\prime}\not \in \mathcal{Q}_1(J)\cup \mathcal{Q}_3(J)$ and $Q^{\prime}\in \mathcal{Q}_2(J)$. Hence, $J\subset Q^{\prime}$ and $J\subset \bigcap_{Q^{\prime}\subset \sqrt{Q}} Q^{\prime}=IK[X]_{\sqrt{Q}}\cap K[X]$.
				
		Trivially, (2-a) implies (2-b) since $J\subset \sqrt{J^m}\subset  \sqrt{IK[X]_{\sqrt{Q}}\cap K[X]}$. Show (2-b) implies (2-a). For $Q\in \mathcal{Q}_2(J)\cup \mathcal{Q}_3(J)$, let $m_Q=\min \{m\mid J^m\subset Q\}$ and $m=\max\{m_Q\mid Q\in \mathcal{Q}_2(J)\cup \mathcal{Q}_3(J)\}.$ Then, $(I:J^{\infty})=(I:J^m)$.  Since $IK[X]_{\sqrt{Q}}\cap K[X]=\bigcap_{Q^{\prime}\in \mathcal{Q},Q^{\prime}\subset \sqrt{Q}} Q^{\prime}$, we obtain $Q^{\prime}\in \mathcal{Q}_2(J)\cup \mathcal{Q}_3(J)$ for any $Q^{\prime}\in \mathcal{Q}$ contained in $\sqrt{Q}$. Thus, we obtain $J^m\subset IK[X]_{\sqrt{Q}}\cap K[X]$. 
	\end{proof}
	
	Using the first saturated quotient, we devise criteria for primary component in Section 4. The second saturated quotient can be used to isolated prime divisor check and generate an isolated primary component in Section 5. The third saturated quotient gives another prime divisor criterion (Criterion 5 in Section 4) other than Corollary \ref{cri2} by the following proposition. 
	
	\begin{Proposition} \label{radthird}
		Let $I$ and $J$ be ideals. Then $\sqrt{(I:(I:J^{\infty}))}=\bigcap_{P\in \Ass (I),J\subset P} P.$
	\end{Proposition}
	\begin{proof}
		Let $\mathcal{Q}$ be a primary decomposition of $I$. By Proposition \ref{satcolon} (3), 
		\[
		\sqrt{(I:(I:J^{\infty}))}=\bigcap_{Q\in \mathcal{Q}_2(J)} \sqrt{(Q:\bigcap_{Q^{\prime}\in \mathcal{Q}_1(J)} Q^{\prime} )}\cap  \bigcap_{Q\in \mathcal{Q}_3(J)}\sqrt{(Q :\bigcap_{Q^{\prime}\in \mathcal{Q}_1(J)} Q^{\prime} )}.
		\]
		Since $\mathcal{Q}$ is minimal, we obtain $Q\not \supset \bigcap_{Q^{\prime}\in \mathcal{Q}_1(J)} Q^{\prime}$ for any $Q\in \mathcal{Q}_2(J)$ and $Q\not \supset \bigcap_{Q^{\prime}\in \mathcal{Q}_1(J)} Q^{\prime} $ for any $Q\in \mathcal{Q}_3(J)$. Thus, by Lemma \ref{cri2}, 
		\begin{align*}
				\sqrt{(I:(I:J^{\infty}))}&=\bigcap_{Q\in \mathcal{Q}_2(J)} \sqrt{(Q:\bigcap_{Q^{\prime}\in \mathcal{Q}_1(J)} Q^{\prime} )}\cap  \bigcap_{Q\in \mathcal{Q}_3(J)}\sqrt{(Q :\bigcap_{Q^{\prime}\in \mathcal{Q}_1(J)} Q^{\prime} )}	\\
								&=\bigcap_{Q\in \mathcal{Q}_2(J)} \sqrt{Q}\cap \bigcap_{Q\in \mathcal{Q}_3(J)}\sqrt{Q}=\bigcap_{P\in \Ass (I),J\subset P} P.  
		\end{align*} 	
	\end{proof}

%==== Section 4 ====================================
\section{Criteria for Primary Component and Prime Divisor}\label{sec-5}
In this section, we present several criteria for primary component which check if a $P$-primary ideal $Q$ is a primary component of $I$ or not without computing primary decomposition of $I$  based on the first saturated quotient. We first propose a general criterion applicable to any primary ideals. Later, we propose some specialized criteria aiming for isolated primary components and maximal ones. Finally, we add criteria for prime divisors. 

\subsection{General Primary Component Criterion}

	We use the first saturated quotient to check if a given primary ideal is a component or not. We introduce a key notion {\em saturated quotient invariant}. 
	\begin{Definition}
		Let $I$ and $J$ be ideals. We say that $J$ is saturated quotient invariant of $I$ if $(I:(I:J)^{\infty})=J.$
	\end{Definition}
	
	Any localization is saturated quotient invariant. Conversely, any proper saturated quotient invariant ideal is some localization of $I$. 
	
	\begin{Lemma} \label{satcoloninv}
		Let $I$ be an ideal and $J$ a proper ideal of $K[X]$. Then, the following conditions are equivalent. 
		
			$(A)$ $J=IK[X]_S\cap K[X]$ for some multiplicatively closed set $S$. 
			
			$(B)$ $J$ is saturated quotient invariant of $I$. 
	\end{Lemma}
	
	\begin{proof}
		Let $\mathcal{Q}$ be a primary decomposition. Show $(A)$ implies $(B)$. From Proposition \ref{satcolon} (1), 
		\begin{equation}
		(I:(I:IK[X]_S\cap A)^{\infty})=\bigcap_{Q\in \mathcal{Q}, IK[X]_S\cap K[X] \subset IK[X]_{\sqrt{Q}}\cap K[X]} Q. \label{eq:1}
		\end{equation}
		By Lemma \ref{mulp}, $IK[X]_S\cap K[X]\subset IK[X]_{\sqrt{Q}}\cap K[X]$ if and only if $Q\cap S=\emptyset $. Thus, 
		\begin{equation}
		\bigcap_{Q\in \mathcal{Q},IK[X]_S\cap K[X]\subset IK[X]_{\sqrt{Q}}\cap K[X]} Q=\bigcap_{Q\in \mathcal{Q},Q\cap S=\emptyset } Q, \label{eq:2}
		\end{equation}
		
		Combining $\eqref{eq:1}$, $\eqref{eq:2}$ and $IK[X]_S\cap K[X]=\bigcap_{Q\in \mathcal{Q},Q\cap S=\emptyset } Q$ by Remark \ref{rm:locpri}, we obtain $(I:(I:IK[X]_S\cap A)^{\infty})=IK[X]_S\cap K[X]$. 
		
		Next, show $(B)$ implies $(A)$. From Proposition \ref{satcolon} (1), 
		\begin{equation}
		(I:(I:J)^{\infty})=\bigcap_{J \subset IK[X]_{\sqrt{Q}}\cap K[X]} Q=J. \label{eq:3}
		\end{equation}
		Let $\mathcal{P}=\{\sqrt{Q}\mid Q\in \mathcal{Q},J \subset IK[X]_{\sqrt{Q}}\cap K[X]\}$. We may assume $\mathcal{P}\neq \emptyset $, otherwise $\mathcal{P}=\emptyset $ and $J=K[X]$. Then $\mathcal{P}$ is {\em isolated} since if $P^{\prime}\in \Ass (I)$ and $P^{\prime}\subset P$ for some $P\in \mathcal{P}$, then $J\subset  IK[X]_{P}\cap K[X]\subset  IK[X]_{P^{\prime}}\cap K[X]$ and $P^{\prime}\in \mathcal{P}$. Let $S=K[X]\setminus \bigcup_{P\in \mathcal{P}} P$. By Lemma \ref{isomul}, $IK[X]_S\cap K[X]=\bigcap_{Q\in \mathcal{Q}, \sqrt{Q}\in \mathcal{P}} Q=\bigcap_{J \subset IK[X]_{\sqrt{Q}}\cap K[X]} Q$. By \eqref{eq:3}, we obtain $IK[X]_S\cap K[X]=J$. 
		
	\end{proof}

	Based on Lemma \ref{satcoloninv}, we have the following criterion for primary component. 
	
	\begin{Theorem}[Criterion 1] \label{primarycri}
		Let $I$ be an ideal and $P$ a prime divisor of $I$. For a $P$-primary ideal $Q$, if $Q\not \supset (I:P^{\infty})$, then the following conditions are equivalent. 
		
			$(A)$ $Q$ is a $P$-primary component for some primary decomposition of $I$.
			
			$(B)$ $(I:P^{\infty})\cap Q$ is saturated quotient invariant of $I$. 
	\end{Theorem}
	
	\begin{proof}
		Show $(A)$ implies $(B)$. Let $\mathcal{Q}$ be a primary decomposition. Let $\mathcal{P}=\{P^{\prime}\in \Ass (I)\mid P\not \subset P^{\prime} \text{ or } P^{\prime}=P \}$ and $S=K[X]\setminus \bigcup_{P^{\prime}\in \mathcal{P}} P^{\prime}$. Then $S$ is a multiplicatively closed set and $(I:P^{\infty})\cap Q\subset IK[X]_S\cap K[X]$ since $(I:P^{\infty})\cap Q=\bigcap_{Q^{\prime}\in \mathcal{Q},P\not \subset \sqrt{Q^{\prime}}} Q^{\prime}\cap Q$. For each $Q^{\prime}\in \mathcal{Q}$ with $Q^{\prime}\cap S=\emptyset $, there is $P^{\prime}\in \mathcal{P}$ such that $\sqrt{Q^{\prime}}\subset P^{\prime}$, i.e. $\sqrt{Q^{\prime}}\in \mathcal{P}$. Thus, $(I:P^{\infty})\cap Q\supset IK[X]_S\cap K[X]$ and $(I:P^{\infty})\cap Q=IK[X]_S\cap K[X].$
		By Lemma \ref{satcoloninv}, $IK[X]_S\cap K[X]$ is saturated quotient invariant of $I$. 

		Show $(B)$ implies $(A)$. By Lemma \ref{satcoloninv}, there is a multiplicatively closed set $S$ such that $(I:P^{\infty})\cap Q=IK[X]_S\cap K[X]$. 
		Let $\mathcal{Q}$ be a primary decomposition of $I$. We know $IK[X]_S\cap K[X]=\bigcap_{Q^{\prime}\in \mathcal{Q},Q^{\prime}\cap S=\emptyset } Q^{\prime}$. By the assumption, $Q\not \supset (I:P^{\infty})$ and thus $(I:P^{\infty})\cap Q$ has a $P$-primary component. Then neither $\bigcap_{Q^{\prime}\in \mathcal{Q},Q^{\prime}\cap S\neq \emptyset } Q^{\prime}$ nor $(I:P^{\infty})$ has a $P$-primary component. Hence, 
		\begin{center}
		$
		I=(I:P^{\infty})\cap Q\cap \bigcap_{Q^{\prime}\in \mathcal{Q},Q^{\prime}\cap S\neq \emptyset } Q^{\prime}=\bigcap_{Q^{\prime}\in \mathcal{Q},P\not \subset \sqrt{Q^{\prime}}}Q^{\prime} \cap Q \cap \bigcap_{Q^{\prime}\in \mathcal{Q},Q^{\prime}\cap S\neq \emptyset } Q^{\prime}
		$
		\end{center}
	 	is a primary decomposition and $Q$ is its $P$-primary component. 
	\end{proof}
	
\subsection{Other Criteria for Primary Component}

	Next, we propose criteria for primary components having special properties which can be applied for particular prime divisors. These criteria may be computed more easily than the general one. 

	\subsubsection{Criterion for Isolated Primary Component:}
	
	If $Q$ is a primary ideal whose radical is an isolated divisor $P$ of an ideal $I$, then we don't need to compute $(I:P^{\infty})$ since the $P$-primary component of $I$ is the localization of $I$ by $P$.  
	
	\begin{Theorem}[Criterion 2] \label{criiso-1}
	Let $I$ be an ideal and $P$ an isolated prime divisor of $I$. For a $P$-primary ideal $Q$, the following conditions are equivalent. 
	
		$(A)$ $Q$ is the isolated $P$-primary component of $I$. 
		
		$(B)$ $(I:(I:Q)^{\infty})=Q$.
	\end{Theorem}
	
	\begin{proof}
		Show $(A)$ implies $(B)$. Let $S=K[X]\setminus P$. By Lemma \ref{satcoloninv}, $Q=IK[X]_S\cap K[X]$ is saturated quotient invariant of $I$ and thus $(I:(I:Q)^{\infty})=Q$. Next, we show $(B)$ implies $(A)$. By Lemma \ref{satcoloninv}, there is a multiplicatively closed set $S$ s.t. $IK[X]_S\cap K[X]=Q$. Since $Q$ is primary, $IK[X]_S\cap K[X]$ is the isolated $P$-primary component. 
	\end{proof}
	
	\subsubsection{Criterion for Maximal Primary Component:} Each isolated prime divisor is minimal in $\Ass (I)$. On the contrary, we consider "maximal prime divisor" and propose the following criterion for it. 
	
	\begin{Definition}
		Let $P$ be a prime divisor of $I$. We say $P$ is {\rm maximal} if there is no prime divisor $P^{\prime}$ of $I$ containing $P$ properly.  
	\end{Definition}
	
	\begin{Theorem}[Criterion 3]
		Let $I$ be an ideal and $P$ a maximal prime divisor of $I$. For $P$-primary ideal $Q$, the following conditions are equivalent. 

		$(A)$ $Q$ is a $P$-primary component of $I$. 
		
		$(B)$ $(I:P^{\infty})\cap Q=I$.
	\end{Theorem}
	\begin{proof}
		Show $(A)$ implies $(B)$. Let $\mathcal{Q}$ be a primary decomposition of $I$ with $Q\in \mathcal{Q}$. Since $P$ is maximal in $\Ass (I)$, $(I:P^{\infty})=\bigcap_{Q^{\prime}\in \mathcal{Q},\sqrt{Q^{\prime}}\not \supset P} Q^{\prime}=\bigcap_{Q^{\prime}\in \mathcal{Q},Q^{\prime}\neq Q} Q^{\prime}$. Thus, $(I:P^{\infty})\cap Q=\bigcap_{Q^{\prime}\in \mathcal{Q},Q^{\prime}\neq Q} Q^{\prime} \cap Q=I$. Next, we show  $(B)$ implies $(A)$. Let $\mathcal{Q}^{\prime}$ be a primary decomposition of $(I:P^{\infty})$. Since $\mathcal{Q}^{\prime}$ does not have $P$-primary component, $\mathcal{Q}^{\prime}\cup \{Q\}$ is a primary decomposition of $I$. 
	\end{proof}	
	
	\subsubsection{Criterion for Another General Primary Component:}  The general case can be reduced to maximal case via localization by maximal independent set (See \cite{greuel2002singular} the definition of maximal independent and its computation). Letting $S =K[U]^{\times}= K[U] \setminus \{0\}$, we obtain the following as a special case of Lemma \ref{prilocal}.
	
	\begin{Theorem}[Criterion 4] \label{maxlocal}
		Let $I$ be an ideal and $P$ a prime divisor of $I$.  If $U$ is a maximal independent set of $P$ in $X$ and $Q$ is a $P$-primary ideal , then the following conditions are equivalent. 
		
			$(A)$ $Q$ is a primary component of $I$.
			
			$(B)$ $Q$ is a primary component of $IK[X]_{K[U]^{\times}}\cap K[X]$.
	\end{Theorem}
		
		\subsection{Additional Criterion for Prime Divisor}
	
	Here, we add a criterion for prime divisor based on the third saturated quotient. 
	
	\begin{Theorem}[Criterion 5] \label{DIPC-2}
		Let $I$ be an ideal and $P$ a prime ideal. Then, the following conditions are equivalent. 
			
			$(A)$ $P\in \Ass (I)$.
			
			$(B)$ $P\supset (I:(I:P))$.
			
			$(C)$ $P\supset (I:(I:P^{\infty}))$.
	\end{Theorem}
	
	\begin{proof}
		By Corollary \ref{DIPC}, $(A)$ is equivalent to $(B)$. By Proposition \ref{radthird}, \\
		$\sqrt{(I:(I:P))}=\sqrt{(I:(I:P^{\infty}))}=\bigcap_{P^{\prime}\in \Ass (I),P\subset P^{\prime}} P^{\prime}$. Thus, equivalence between $(A)$ and $(C)$ is proved by the similar way of Corollary \ref{DIPC}. 
	\end{proof}
	
	Next, we devise criteria for isolated prime divisor based on the second saturated quotient. 
	
		\begin{Lemma}\label{satpseudo}
		Let $I$ be an ideal and $P$ an isolated prime divisor of $I$. If $\overline{Q}$ is the $P$-pseudo-primary component of $I$, then $(I:(I:P^{\infty})^{\infty})=\overline{Q}$.
	\end{Lemma}
	\begin{proof}
			Let $\mathcal{Q}$ be a primary decomposition of $I$. By Proposition \ref{satcolon} (2), 
			\begin{center}
			$(I:(I:P^{\infty})^{\infty})=\bigcap_{Q\in \mathcal{Q},P\subset \sqrt{IK[X]_{\sqrt{Q}}\cap K[X]}} Q.$
			\end{center}
		Thus it is enough to show that the following statements are equivalent for each $Q\in \mathcal{Q}$. 
			
			(1-a) $P\subset \sqrt{IK[X]_{\sqrt{Q}}\cap K[X]}$.
			
			(1-b) $P$ is the unique isolated prime divisor which is contained in $\sqrt{Q}$.\\
		Show (1-a) implies (1-b). As $ \sqrt{IK[X]_{\sqrt{Q}}\cap K[X]}\subset \sqrt{Q}$, we know $P\subset \sqrt{Q}$. Then, suppose there is another isolated prime divisor $P^{\prime}$ contained in $\sqrt{Q}$. We obtain
		\[
		 \sqrt{IK[X]_{\sqrt{Q}}\cap K[X]}=\bigcap_{Q^{\prime}\in \mathcal{Q},Q^{\prime}\subset \sqrt{Q}} \sqrt{Q^{\prime}}\subset P^{\prime}.
		\]
		However, this implies $P\subset P^{\prime}$ and contradicts that $P^{\prime}$ is isolated. It is easy to prove that (1-b) implies (1-a). 
	\end{proof}
	\begin{Theorem}[Criterion 6] \label{eqiso}
		Let $I$ be an ideal and $P$ a prime ideal containing $I$. Then, the following conditions are equivalent. 

			$(A)$ $P$ is an isolated prime divisor of $I$.
			
			$(B)$ $(I:(I:P^{\infty})^{\infty})\neq K[X]$.
	\end{Theorem}
	
	\begin{proof}
		Show $(A)$ implies $(B)$. By Lemma \ref{satpseudo}, $(I:(I:P^{\infty})^{\infty})=\overline{Q}\neq K[X]$. Show $(B)$ implies $(A)$. By Proposition \ref{satcolon} (2), 
		\begin{center}
		$(I:(I:P^{\infty})^{\infty})=\bigcap_{Q\in \mathcal{Q},P\subset \sqrt{IK[X]_{\sqrt{Q}}\cap K[X]}} Q\neq K[X]$
		\end{center}
		for a primary decomposition $\mathcal{Q}$ of $I$. Then, there is an isolated prime divisor $P^{\prime}$ containing $P$. Since $\sqrt{I}\subset P\subset P^{\prime}$ and $P^{\prime}$ is isolated, this implies $P=P^{\prime}$ is isolated. 
	\end{proof}
	
	Since each prime divisor of $I$ contains $I$, Theorem \ref{eqiso} directly induces the
following.
	
	\begin{Corollary}[Criterion 7]
		Let $I$ be an ideal and $P$ a prime divisor of $I$. Then, 
			
			{\rm (i)} $P$ is isolated if $(I:(I:P^{\infty})^{\infty})\neq K[X]$,
		
			{\rm (ii)} $P$ is embedded if $(I:(I:P^{\infty})^{\infty})=K[X]$.
	\end{Corollary}

%==== Section 5 ====================================
\section{Local Primary Algorithm}
In this section, we devise Local Primary Algorithm (LPA) which computes $P$-primary component of $I$. Our method applies different procedures for two cases; isolated and embedded. Algorithm 1 shows the outline of LPA. Its termination comes from Proposition \ref{hullpm}. We remark that, for given prime divisors disjoint from a multiplicatively closed set $S$, we can compute all primary components disjoint from $S$ by LPA. Then their intersection gives the localization by $S$. 
	
	\subsection{Generating Primary Component}
	
	First, we introduce several ways to generate primary component through
equidimensional hull computation.
	\begin{Proposition}[\cite{Eisenbud1992}, Section 4. \cite{matzat2012algorithmic}, Remark 10]  \label{hullpm}
	Let $I$ be an ideal and $P$ a prime divisor of $I$. For any positive integer m, $I+P^m$ is P-hull-primary, and for a sufficiently large integer $m$, $\hull (I+P^m)$ is a $P$-primary component appearing in a primary decomposition of $I$. 
	\end{Proposition}
	
	We can use Criteria for Primary Component to check $m$ is large enough or not. If $P$ is an isolated prime divisor, then the component is computed directly by using the second saturated quotient. By Lemma \ref{psuedo-primary-component} and Lemma \ref{satpseudo}, we obtain the following theorem. 
	
	\begin{Theorem} \label{thm-p}
	Let $I$ be an ideal and $P$ an isolated prime divisor of $I$. Then
	\[
	\hull ((I:(I:P^{\infty})^{\infty}))
	\]
	is the isolated $P$-primary component of $I$. 
	\end{Theorem}

	\begin{algorithm}[H]
		\begin{algorithmic}[1]
			\REQUIRE{$I$: an ideal, $P$: a prime ideal}
			\ENSURE $
				\begin{cases}
					\text{a $P$-primary component of $I$ if $P$ is a prime divisor of $I$}\\
					\text{"$P$ is not a prime divisor" otherwise}
				\end{cases}
				$
			\IF{$P$ is a prime divisor of $I$ ({\bf Criterion 5})} 
				\IF{$P$ is isolated ({\bf Criteria 6,7})} 
						\STATE $\overline{Q}\gets $ the $P$-pseudo-primary component of $I$  ({\bf Lemma  \ref{satpseudo}})
						\STATE $Q\gets \hull (\overline{Q})$ ({\bf Theorem \ref{thm-p}})
					\RETURN $Q$ is the isolated $P$ primary component
				\ELSE 
					\STATE $m\gets 1$
					\WHILE{$Q$ is not primary component of $I$ ({\bf Criteria 1,3,4})}
						\STATE $\overline{Q}\gets $ a $P$-hull-primary ideal related to $m$ ({\bf Proposition \ref{hullpm}, Lemma \ref{pm}})
						\STATE $Q\gets \hull (\overline{Q})$
						\STATE{$m\gets m+1$}
					\ENDWHILE
					\RETURN $Q$ is an embedded $P$-primary component
				\ENDIF
				
			\ELSE
				\RETURN "$P$ is not a prime divisor"
			\ENDIF	
		\end{algorithmic}
		\caption{General Frame of Local Primary Algorithm}
	\end{algorithm}
	
	\subsection{Techniques for Improving LPA}
	
	We introduce practical technique for implement LPA.  
	
	\subsection{Another Way of Generating Primary Component}
	
	Let $G=\{f_1,\ldots,f_r\}$ be a generator of $P$. Usually we take $\{f_1^{e_1}f_2^{e_2}\cdots f_r^{e_r}\mid e_1+\cdots +e_r=m\}$ as a generator of $P^m$ for a positive integer $m$. However, this generator has $\frac{(r+m-1)!}{(r-1)!m!}$ elements and it becomes difficult to compute $\hull (I+P^m)$ when $m$ becomes large. To avoid the explosion of the number of the generator, we can use $P_G^{[m]}=(f_1^m,\ldots,f_r^m)$ instead. 
	
	\begin{Lemma} \label{minpri}
		Let $\mathcal{Q}$ be a primary decomposition of $I$ and $Q\in \mathcal{Q}$. If $\sqrt{Q}$-hull-primary ideal $Q^{\prime}$ satisfies $I\subset Q^{\prime}\subset Q$, 
		then $(\mathcal{Q}\setminus \{Q\})\cup \{\hull (Q^{\prime})\}$ is another primary decomposition of $I$. 
	\end{Lemma}
	\begin{proof}
		By Lemma \ref{hull-primary subset}, we obtain $I\subset Q^{\prime} \subset \hull (Q^{\prime})\subset Q$. Since $I\cap \hull (Q^{\prime})=I$ and $Q\cap \hull (Q^{\prime})=\hull (Q^{\prime})$, we obtain
		\[
		I=I\cap \hull (Q^{\prime})=\left(\bigcap_{Q^{\prime \prime}\in \mathcal{Q},Q^{\prime \prime}\neq Q}Q^{\prime \prime} \cap Q\right)\cap \hull (Q^{\prime})=\bigcap_{Q^{\prime \prime}\in \mathcal{Q},Q^{\prime \prime}\neq Q}Q^{\prime \prime} \cap \hull (Q^{\prime}).
		\]
		Thus, $(\mathcal{Q}\setminus \{Q\})\cup \{\hull (Q^{\prime})\}$ is an irredundant primary decomposition of $I$. 
	\end{proof}

	\begin{Lemma} \label{pm}
		For any positive integer $m$, $I+P_G^{[m]}$ is $P$-hull-primary, and for a sufficiently large $m$, $\hull (I+P_G^{[m]})$ is a $P$-primary component appearing in a primary decomposition of $I$. 
	\end{Lemma}
	\begin{proof}
	As $\sqrt{I+P}=\sqrt{I+P_G^{[m]}}=P$, $I+P_G^{[m]}$ is $P$-hull-primary. By Theorem \ref{hullpm}, $\hull (I+P^m)$ is a $P$-primary component of $I$ for a sufficiently large $m$. Since $I\subset I+P_{G}^{[m]}\subset I+P^m\subset \hull (I+P^m)$, $\hull (I+P_{G}^{[m]})$ is a $P$-primary component by Lemma \ref{minpri}. 
	\end{proof}
		
	\subsection{Equidimensional Hull Computation with MIS}
	
	Next, we devise another computation of $\hull (I+P^m)$ based on {\em maximal independent set} (MIS) which is much efficient than computations based on Proposition \ref{comhull}. Similarly, by this technique we can replace $I$ with $IK[X]_{K[U]^\times}\cap K[X]$ 
at the first step of LPA.

	\begin{Lemma} \label{maxhull}
		Let $I$ be a $P$-hull-primary ideal. For a maximal independent set $U$ of $P$, $\hull (I)=IK[X]_{K[U]^{\times}}\cap K[X]$.
	\end{Lemma}
	\begin{proof}
		Let $\mathcal{Q}$ be a primary decomposition of $I$. Then, $\hull (I)$ is the unique primary component disjoint from $K[U]^{\times}$. Thus, 
		\begin{center}
		$IK[X]_{K[U]^{\times}}\cap K[X]=\bigcap_{Q\in \mathcal{Q}, Q\cap K[U]^{\times}=\emptyset } Q=\hull (I).$ 
		\end{center} 
	\end{proof}

%==== Section 6 ====================================
\section{Experiments}
	We made a preliminary implementation on a computer algebra system Risa/Asir \cite{risaasir} and apply it to several examples as naive experiments.  Here we show some typical examples. Timings are measured on a PC with Xeon E5-2650 CPU. 
	
	First, we see an ideal whose embedded primary components are hard to compute. Let $I_1(n)=(x^2)\cap (x^4,y)\cap (x^3,y^3,(z+1)^n+1)$. If $n$ is considerable large, it is difficult to compute a full primary decomposition of $I_1(n)$ though the isolated devisor $(x)$ can be detected pretty easily. We apply Local Primary Algorithm (LPA) for this example to compute the isolated primary component for $P_1=(x)$. We also see another example which is more valuable for mathematics. An ideal $A_{k,m,n}$ is defined in \cite{strumfels} and its primary decomposition has important meanings in Computer Algebra for Statistics. We consider an isolated prime divisor $P_2=(x_{13},x_{23},x_{33},x_{43})$ of $A_{3,4,5}$ in $\mathbb{Q}[x_{ij}\mid 1\le i\le 4,1\le j\le 5]$. In Table 1, we can see LPA has certain effectiveness by its speciality comparing a full primary decomposition function noro\_pd.syci\_dec. From Proposition \ref{comhull}, we also use double ideal quotient to compute equidimensional hull.   
	
	\begin{table}[H]
		\begin{center}
			\begin{tabular}{ccccccc}\hline
Algorithm&$I_1(100)$&$I_1(200)$ &$I_1(300)$&$I_1(400)$ &$I_1(500)$&$A_{3,4,5}$/$P_2$ \\\hline
noro\_pd.syci\_dec & 0.36&15.6&88.3&289&96.0&$>$3600\\\hline
LPA & 0.02&0.04&0.07&0.11&0.14 &14.3\\\hline\\
\end{tabular}
			\caption{Local Primary Algorithm (Isolated)} \label{table1}
		\end{center} 
	\end{table}

	Second, we consider embedded prime divisors; {\small $P_3=(x_{12}x_{31}-x_{32}x_{11},x_{42}x_{11}-x_{41}x_{12},x_{42}x_{31}-x_{41}x_{32},x_{44}x_{31}-x_{41}x_{34},x_{44}x_{32}-x_{42}x_{34},x_{13},x_{21},x_{22},x_{23},x_{24},x_{33},x_{43})$} of $A_{2,4,4}$ in $\mathbb{Q}[x_{ij}\mid 1\le i\le 4,1\le j\le 4]$ and {\small $P_4=(x_{16}x_{27}-x_{17}x_{26},$ $ x_{34}x_{13}-x_{33}x_{14},x_{37}x_{16}-x_{36}x_{17},x_{36}x_{27}-x_{37}x_{26},x_{12},x_{15},x_{21},x_{22},x_{23},x_{24},x_{25},x_{32},x_{35})$} of $A_{2,3,7}$ in $\mathbb{Q}[x_{ij}\mid 1\le i\le 3,1\le j\le 7]$. In Table 2, LPA-Pm is an implementation based on Lemma \ref{pm} and LPA-MIS is one from Lemma \ref{maxhull} and Criteria 3, 4. Both methods are implemented in  LPA-(Pm+MIS).  The primitive LPA is not practical since the cost of computing $\hull (I+P^m)$ is much high. On the other hand, we can see LPA-(Pm+MIS) has good effectiveness by its speciality comparing a full primary decomposition function noro\_pd.syci\_dec. 
	
\begin{table}[H]
	\begin{center}
		\begin{tabular}{lcc}\hline
Algorithm& $A_{2,4,4}$/$P_3$& $A_{2,3,7}$/$P_4$\\\hline
noro\_pd.syci\_dec & 3.11 &34.8\\\hline
LPA & $>3600$ &168\\\hline
LPA-Pm & 4.75& 29.1\\\hline
LPA-MIS & 0.58 &0.38\\\hline
LPA-(Pm+MIS) & 0.15& 0.08\\\hline\\
\end{tabular}
		\caption{Local Primary Algorithm (Embedded) and its Improvement} \label{table2}
	\end{center} 
\end{table}

%==== Section 7 ====================================
\section{Conclusion and Future Work}
In commutative algebra, the operation of "localization by a prime ideal" is well-known as a basic tool. However, its computation through primary decomposition is much difficult. Thus, we devise a new effective localization {\em Local Primary Algorithm} (LPA) using Double Ideal Quotient(DIQ) and its variants without computing unnecessary primary components for localization. For our construction of LPA, we devise several criteria for primary component based on DIQ and its variants. We take preliminary benchmarks for some examples to examine certain effectiveness of LPA coming from its speciality. To make our LPA very practical we shall continue to improve it through obtaining timing data for a lot of larger examples.

In future work, we are finding a way to compute "sample points" of prime divisors. For localization it does not need all divisors; it is enough to find $f_P\in P\cap S$ for each prime divisor $P$ with $P\cap S\neq \emptyset$ and we obtain $IK[X]_S\cap K[X]=(I:(\prod_{P\cap S\neq \emptyset } f_P)^{\infty})$. Another work is to apply our primary component criteria to {\it probabilistic or inexact} methods for primary decomposition, such as numerical ones. Probabilistic or inexact ways have low computational costs, however, they have low accuracy for outputs. Hence, our criterion using double ideal quotient may help to guarantee their outputs. Finally, localization in general setting, that is localization by a prime ideal not necessary associated is interesting work. 

%==== Acknowledgement ====================================
\medskip
\noindent
{\bf Acknowledgement:}  The authors would like to thank the referees for their helpful comments to improve the presentation of this paper.  The authors are also grateful to Masayuki Noro for technical assistance with the computer experiments and coding on Risa/Asir. 

%==== Bibliography =================================
\bibliographystyle{plain}

\end{document}